\documentclass[reqno,11pt,centertags]{article}
\usepackage{amsmath,amsthm,amscd,amssymb,latexsym,upref}
\date{\today}
\usepackage[T2A,OT1]{fontenc}
\usepackage[ot2enc]{inputenc}
\usepackage[russian,english]{babel}
\usepackage[margin=3.5 cm]{geometry}

\usepackage{hyperref}
\hypersetup{
    colorlinks=true, %set true if you want colored links
    linktoc=all,     %set to all if you want both sections and subsections linked
    linkcolor=blue,  %choose some color if you want links to stand out
}

\newcommand{\bbD}{{\mathbb{D}}}

\newcommand{\bbT}{{\mathbb{T}}}

\newcommand{\cP}{{\mathcal{P}}}

\newcommand{\z}{\zeta}

\renewcommand{\Im}{\text{\rm Im}\,}

\allowdisplaybreaks \numberwithin{equation}{section}
\newtheorem{theorem}{Theorem}[section]

\newtheorem{lemma}[theorem]{Lemma}
\newtheorem{proposition}[theorem]{Proposition}

\newtheorem{corollary}[theorem]{Corollary}

\theoremstyle{definition}
\newtheorem{definition}[theorem]{Definition}

\newtheorem{remark}[theorem]{Remark}

\def\be{\begin{equation}}
\def\ee{\end{equation}}
\def\bea{\begin{eqnarray}}
\def\eea{\end{eqnarray}}
\def\bean{\begin{eqnarray*}}
\def\eean{\end{eqnarray*}}

\def\restr#1{\,\vrule\,\lower1ex\hbox{$#1$}}

\def\a{\alpha}
\def\b{\beta}
\def\d{\delta}
\def\D{\Delta}

\def\g{\gamma}
\def\G{\Gamma}

\def\z{\zeta}

\title
{Automorphic Carath\' eodory-Julia Theorem}
\author{Alexander Kheifets}

\begin{document}

\maketitle

\begin{center}
\textit{Dedicated to the memory of my teacher Prof. Victor Emmanuilovich Katsnelson}
\end{center}
\medskip

%%% \tableofcontents

\begin{abstract}
Let $w(\z)$ be a function analytic on $\bbD$, $|w(\z)|\le 1$. Let $|t_0|=1$.
Assume that $w$ and $w'$ have nontangential boundary values $w_0$ and $w'_0$,
respectively, at $t_0$, $|w_0|=1$. Then (Carath\'eodory - Julia)
$t_0\dfrac{w'_0}{w_0}\ge 0$. The goal of this paper is to obtain a lower bound on
this ratio if $w$ is character-automorphic with respect to a Fuchsian group
(Theorem \ref{230606-07}).

\end{abstract}

\section{Introduction.}
A very particular case of the classical Nevanlinna-Pick problem states that for every point $\z_0$ of the open unit disk $\bbD$ ($|\z_0|<1$) and for every
complex number $w_0$, $|w_0|\le 1$ there exists function $w(\z)$ analytic on $\bbD$, $|w(\z)|\le 1$ such that $w(\z_0)=w_0$.

\medskip\noindent
A character automorphic analogue of this theorem was proved in \cite{Kup-Yud-1997}.
\begin{definition}
Conformal maps of the unit disk $\bbD$ onto itself
are of the form
$$
\zeta\to\frac{a\zeta+b}{\overline b\zeta+\overline a},\quad |a|^2-|b|^2=1.
$$
%%%Equivalently the transformation can be written as
%%%$$
%%%\zeta\to\frac{\zeta-\zeta_0}{1-\zeta\overline \zeta_0}c,\quad |\zeta_0|<1,\quad |c|=1.
%%%$$
They form a group under composition.
Discrete subgroups of this group are called {\bf Fuchsian groups}. By $\G^*$ we will
denote the group of unitary characters of $\G$.
\end{definition}
Automorphic analogue of the Nevanlinna-Pick problem reads as follows: given Fuchsian
group $\G$ and a unitary character $\b$ of this group. Given point $\z_0\in\bbD$ and number $w_0$.
Find $\b$ {\bf automorphic} analytic on $\bbD$ function $w(\z)$, that is,
$$
w(\g(\z))=\b(\g)w(\z),\quad\forall \g\in\G,\quad |\z|<1,
$$
such that $|w(\z)|\le 1$ and $w(\z_0)=w_0$.

It was shown in \cite{Kup-Yud-1997} that in this case the constraint on $w_0$
is more restrictive than just $|w_0|\le 1$. However, even before looking for this constraint one needs to know that $\b$-automorphic
bounded analytic functions exist at all.

\medskip\noindent
Widom gave a remarkable characterization of groups $\G$ such that
for every unitary character $\a$ of $\G$ there exists a non-zero
$\a$-automorphic bounded analytic function. We present here one of the equivalent forms
of this theorem.
\begin{theorem}[Widom-Pommerenke \cite{Widom71}, \cite{Pom}]
Assume that group $\G$ is of convergent type, that is the Blaschke product over the orbit of $\G$ converges
$$
g_{\z_0}(\zeta)=\prod\limits_{\gamma\in\Gamma}
\dfrac{\zeta-\gamma(\zeta_0)}{1- \zeta\overline{\gamma(\zeta_0)}}C_\gamma
=\prod\limits_{\gamma\in\Gamma}
\dfrac{\gamma(\zeta)-\zeta_0}{1- \gamma(\zeta)\overline\zeta_0}\widetilde C_\gamma .
$$
(This function is called the Green function of group $\G$). Also assume that $\G$ does not contain elliptic elements.
Then for every unitary character $\a$ of $\G$ there exists a non-zero
$\a$-automorphic bounded analytic function {\bf if and only if} $g'_{\z_0}(\zeta)$ is a function of bounded characteristic.
\end{theorem}
Fuchsian group $\G$ is said to be of {\bf Widom type} if $g'_{\z_0}(\zeta)$ is of bounded characteristic.
If $g'_{\z_0}(\zeta)$ is of bounded characteristic, then,  by straightforward use of
Poincar\'e theta series associated with $g'_{\z_0}$ (see, e.g. \cite{Pom}), one can show that
for every character $\a$ there exist nonconstant $\a$ automorphic bounded analytic functions.
The converse is a long and subtle story (\cite{Widom71}).

%%%  We need one notation to state an automorphic analogue of the Nevanlinna-Pick theorem.
Let $H^2(\a)$ be the set of all $\a$ automorphic functions in the Hardy space $H^2$.
For groups of Widom type it is a nonzero closed subspace.
We denote by $k^{\a}_{\z_0}$ the orthogonal projection of $\dfrac{1}{1-t\overline\z_0}$ on $H^2(\a)$.
The following theorem is a special case of what was proved in \cite{Kup-Yud-1997}.
\begin{theorem}[Kupin-Yuditskii]\label{230601-02}
Let $\G$ be a Fuchsian group of Widom type. Let $\b$ be a unitary character of $\G$. Let $|\z_0|< 1$.
There exists a $\b$ automorphic function $w(\z)$ analytic
on $\bbD$, $|w(\z)|\le 1$ such that $w(\z_0)=w_0$ if and only if
\be\label{230601-01}
|w_0|^2\le\underset{\a\in\G^*}{\inf}\dfrac{k^{\b\a}_{\z_0}(\z_0)}{k^{\a}_{\z_0}(\z_0)}.
\ee
\end{theorem}
Boundary analogue of the Nevanlinna - Pick theorem is the following classical
\begin{theorem}[Carath\'eodory - Julia]\label{230601-03}
Let $|t_0|=1$. Let $w(\z)$ be function analytic on $\bbD$, $|w(\z)|\le 1$ such that
nontangential boundary values of $w(\z)$ and $w'(\z)$ at point $t_0$ equal $w_0$ and
$w'_0$, respectively; $|w_0|=1$, $w'_0$ is finite. Then
\be\label{230601-04}
t_0\dfrac{w'_0}{w_0}\ge 0.
\ee
Conversely, let $|w_0|=1$ and $w'_0$ be such that \eqref{230601-04} holds.
Then there exists analytic on $\bbD$ function $w(\z)$, $|w(\z)|\le 1$ such that
nontangential boundary values of $w(\z)$ and $w'(\z)$ at point $t_0$ equal $w_0$ and
$w'_0$, respectively.
\end{theorem}
More details about the Carath\'eodory - Julia theory are given in Appendix. In particular
(see \eqref{230606-01} in Appendix) for functions in Theorem \ref{230601-03} we have that
\be\label{230606-02}
\dfrac{w(t)-w_0}{t -t_0}\in H^2.
\ee
The goal of this paper is to give a character-automorphic version of Theorem \ref{230601-03}
in spirit of Theorem \ref{230601-02}. Basically, the constraint on the given values will be more
restrictive than just positivity in \eqref{230601-04}.
However, even before looking for this constraint one at least needs to know
(in view of \eqref{230606-02}) that there exist
$\b$-automorphic functions of the form
\be\label{230606-03}
1+(\z-t_0)h(\z),\quad h\in H^2.
\ee
A sufficient condition for that is given in Theorem \ref{221231_04} and Corollary \ref{230517-04}.
This result to a large extent is motivated
by Widom's result and by paper \cite{Kh-Yud-2019}. It is stated in terms of the Martin function.
\begin{definition}
Let $\G$ be a Fuchsian group. Let $|t_0|=1$ and assume that
\be\label{BK-2019-11-03-03}
\sum\limits_{\gamma\in\Gamma}|\gamma'(t_0)|<\infty.
%%%%%\footnote{It was shown in \cite{Kh-Yud-2019} that this condition is a part of the necessary and sufficient condition
%%%%%of the following property of group $\G$: for every character $\a\in\G^*$ there exists an $\a$ automorphic function $f\ne 0$ on $\bbD$
%%%%%such that $\dfrac{f(\z)}{\z-t_0}$ is in the Hardy space $H^2$. See more in ....}
\ee
We define Martin function of $t_0$ as
%%%%%\be\label{BK-2019-11-03-04}
%%%%%m_{t_0}(\zeta)=i\sum\limits_{\gamma\in\Gamma}
%%%%%\dfrac{\gamma(t_0)+\zeta}{\gamma(t_0)-\zeta}|\gamma'(t_0)|.
%%%%%\ee
%%%%%%%%%%%%%%%%%%%%%%%%%%%%%%%%%%%%%%%%%%%%%%%%%
%%%%\begin{remark}\label{211113_01}
%%%%In what follows, instead of $m_+$, we will use this notation for Martin function
%%%%(Widom group, Akhiezer-Levin $t_0$)
\be\label{211113_02}
m_{t_0}(\z)
=\dfrac{i}{2}\sum\limits_{\gamma\in\Gamma}
\dfrac{\gamma(t_0)+\zeta}{\gamma(t_0)-\zeta}|\gamma'(t_0)|
=\dfrac{i}{2}\sum\limits_{\gamma\in\Gamma}
\left(\dfrac{2\gamma(t_0)}{\gamma(t_0)-\zeta}-1\right)|\gamma'(t_0)|.
\ee
Then
\be\label{211113_03}
m'_{t_0}(\z)
=i\sum\limits_{\gamma\in\Gamma}
\dfrac{\gamma(t_0)|\gamma'(t_0)|}{(\gamma(t_0)-\zeta)^2}
=it_0\sum\limits_{\gamma\in\Gamma}
\dfrac{\gamma'(t_0)}{(\gamma(t_0)-\zeta)^2}.
\ee
The later equality holds since
$$
|\g'(t_0)|=t_0\dfrac{\g'(t_0)}{\g(t_0)}.
$$
The derivative of the Martin function also equals
\be\label{221003_01}
m'_{t_0}(\z)=it_0\sum\limits_{\gamma\in\Gamma}
\dfrac{\gamma'(\z)}{(\gamma(\z)-t_0)^2}
\ee
(see Lemma \ref{230605-01} of Appendix).
\end{definition}
The following theorem was proved in \cite{Kh-Yud-2019} (under some additional assumptions about group $\G$)
\begin{theorem}[\cite{Kh-Yud-2019}]
Let $|t_0|=1$.
For every unitary character $\a$ of group $\G$ there exists an $\a$ automorphic function of the form
\be\label{230606-04}
(\z-t_0)h(\z),\quad h\in H^2
\ee
if and only if
\begin{enumerate}
\item
%%%% \be\label{BK-2019-11-03-03}
$\sum\limits_{\gamma\in\Gamma}|\gamma'(t_0)|<\infty$;
%%%% \ee
\item The derivative of the Martin function $m'_{t_0}(\z)$
is a function of bounded characteristic.
\end{enumerate}
\end{theorem}
The present paper is organized as follows. In Section \ref{230606-05} we discuss equivalent forms and consequences
of the assumption that $m'_{t_0}$ is a function of bounded characteristic. In particular we give a version
of Pommerenke theorem (Theorem \ref{230311-01}) stating that in this case
$$m'_{t_0}=\dfrac{\D_{t_0}}{\varphi},$$
where $\D_{t_0}$ is
an inner function, $\varphi$ is an outer function, $|\varphi|\le 4$. In Theorem \ref{221231_04} and Corollary \ref{230517-04} we
demonstrate that under assumptions \eqref{230312-03} and \eqref{221231_01} for every character $\a$ of the group $\G$ there exists
an $\a$ automorphic function of the form
$$1+(\z-t_0)h(\z),\quad h\in H^2.$$
In Section \ref{230606-06} we define kernels $k^{\a}_{t_0}$ as extremal functions in the above class of functions and we discuss some of their properties.
The kernels are used in
Theorem \ref{230606-07} to get a constraint on $t_0\dfrac{w'_0}{w_0}$ if there exists a $\b$ automorphic
analytic function $w(\z)$ on $\bbD$, $|w(\z)|\le 1$,
such that $w$ and $w'$ have
nontangential boundary values $w_0$, $|w_0|=1$ and $w'_0$, respectively, at the boundary point $t_0$.
In Remark \ref{230623-08} a necessary and sufficient condition is stated for this existence to hold for {\bf every}
character $\b$.

\section{Pommerenke Type Theorems.}\label{230606-05}

The goal of this section is to prove Theorems \ref{17oct14} and \ref{230311-01}.
These theorems are analogues of some parts of Theorem 4 in \cite{Pom}. They were proved
(under some additional assumptions on $\G$) in \cite{Kh-Yud-2019} (Theorems 5.1, 5.4).
We will need two lemmas.
\begin{lemma}\label{20230204-03}
Let $|p|\ge 1$. Then
$$
|p-1|\le 2|p-r|,\quad 0\le r\le 1.
$$
\end{lemma}
\begin{proof}
We write $p$ as
$$
p=A(x+iy),\quad A\ge 1,\quad x^2+y^2=1.
$$
Then
$$
\dfrac{|p-1|^2}{|p-r|^2}=\dfrac{(Ax-1)^2+(Ay)^2}{(Ax-r)^2+(Ay)^2}=\dfrac{A^2-2Ax+1}{A^2-2Arx+r^2},\quad -1\le x\le 1.
$$
The derivative with respect to $x$ is
$$
\dfrac{-2A}{A^2-2Arx+r^2}-\dfrac{A^2-2Ax+1}{(A^2-2Arx+r^2)^2}(-2Ar)
$$
$$
=
\dfrac{-2A}{(A^2-2Arx+r^2)^2}\left[(A^2-2Arx+r^2)-(A^2-2Ax+1)r\right]
$$
$$
=
\dfrac{-2A}{(A^2-2Arx+r^2)^2}\left[A^2+r^2-A^2r-r\right]
$$
$$
=
\dfrac{-2A}{(A^2-2Arx+r^2)^2}(A^2-r)(1-r)\le 0,
$$
since $A\ge 1$ and $r\le 1$. Therefore, the maximum is assumed at $x=-1$. Respectively, $y=0$ and $p=-A$. Then
$$
\dfrac{|p-1|}{|p-r|}=\dfrac{|-A-1|}{|-A-r|}=\dfrac{A+1}{A+r}\le \dfrac{A+1}{A}=1+\dfrac{1}{A}\le 2,
$$
since $r\ge 0$ and $A\ge 1$.
\end{proof}
\begin{lemma}\label{17oct15}
Let $m_{t_0}$ be the Martin function defined in \eqref{211113_02}.
Let $\mathbb T$ be the unit circle and let $t\in\mathbb T$.
A real nontangential limit $m_{t_0}(t)$ exists and a finite nontangential limit $m'_{t_0}(t)$ exists
{\bf if and only if}
\be\label{230528-01}
\sum\limits_{\gamma\in\Gamma}
\dfrac{|\gamma'(t_0)|}{|\gamma(t_0)-t|^2}<\infty.
\ee
In this case
\be\label{230528-02}
m_{t_0}(t)
=\dfrac{i}{2}\sum\limits_{\gamma\in\Gamma}
\dfrac{\gamma(t_0)+t}{\gamma(t_0)-t}|\gamma'(t_0)|
\ee
and
\be\label{211113_06'}
it m'_{t_0}(t)
=\sum\limits_{\gamma\in\Gamma}
\dfrac{|\gamma'(t_0)|}{|\gamma(t_0)-t|^2}.
\ee
\end{lemma}
\begin{proof}
This is Frostman Theorem \ref{230526-02} of Appendix.
%%%\be\label{211113_04}
%%%m'_{t_0}(t)
%%%=it_0\sum\limits_{\gamma\in\Gamma}
%%%\dfrac{\gamma'(t_0)}{(\gamma(t_0)-t)^2}
%%%=it_0\sum\limits_{\gamma\in\Gamma}
%%%\dfrac{\gamma'(t_0)}{(\gamma(t_0)-t)(\gamma(t_0)-t)}
%%%\ee
%%%$$
%%%=it_0\sum\limits_{\gamma\in\Gamma}
%%%\dfrac{\gamma'(t_0)}{t\gamma(t_0)(\overline t-\overline{\gamma(t_0)})(\gamma(t_0)-t)}
%%%=-\dfrac{i}{t}\sum\limits_{\gamma\in\Gamma}
%%%\dfrac{|\gamma'(t_0)|}{|\gamma(t_0)-t|^2}.
%%%$$
%%%This yields \eqref{211113_06'}.
\end{proof}
\begin{theorem} \label{17oct14} $m'_{t_0}$ is of bounded characteristic if and only if
\be\label{230312-01}
\int\limits_{\mathbb T}\log
\left(
\sum\limits_{\gamma\in\Gamma}
\dfrac{|\gamma'(t_0)|}{|\gamma(t_0)-t|^2}
\right)
\mu(dt)<\infty,
\ee
where $\mu(dt)$ is the normalized Lebesgue measure on $\mathbb T$.
\end{theorem}
%%%%%\begin{remark}
%%%%%Since
%%%%%$$
%%%%%\sum\limits_{\gamma\in\Gamma}
%%%%%\dfrac{|\gamma'(t_0)|}{|\gamma(t_0)-t|^2}
%%%%%\ge
%%%%%\dfrac{1}{|t_0-t|^2}\ge\dfrac{1}{4},
%%%%%$$
%%%%%\eqref{230312-01} holds if and only if
%%%%%$$
%%%%%\int\limits_{\mathbb T}\log^+
%%%%%\left(
%%%%%\sum\limits_{\gamma\in\Gamma}
%%%%%\dfrac{|\gamma'(t_0)|}{|\gamma(t_0)-t|^2}
%%%%%\right)
%%%%%\mu(dt)<\infty.
%%%%%$$
%%%%%\end{remark}
\begin{proof}
If $m'_{t_0}$ is of bounded characteristic, then inequality \eqref{230312-01} follows from formula \eqref{211113_06'}.
Conversely, for $0<r<1$
$$
|m'_{t_0}(rt)|
\le
\sum\limits_{\gamma\in\Gamma}
\dfrac{|\gamma'(t_0)|}{|\gamma(t_0)-rt|^2}
\le
4\sum\limits_{\gamma\in\Gamma}
\dfrac{|\gamma'(t_0)|}{|\gamma(t_0)-t|^2}
.
$$
The first inequality is immediate from formula \eqref{211113_03} for $m'_{t_0}$,
the second inequality is due to Lemma \ref{20230204-03} applied to $p=\g(t_0)\overline t$.
Therefore,
$$
\int\limits_{\mathbb T}
\log^+|m'_{t_0}(rt)|\mu(dt)
\le
\int\limits_{\mathbb T}\log^+
\left(
4\sum\limits_{\gamma\in\Gamma}
\dfrac{|\gamma'(t_0)|}{|\gamma(t_0)-t|^2}
\right)
\mu(dt)
$$
$$
=
\int\limits_{\mathbb T}\log
\left(
4\sum\limits_{\gamma\in\Gamma}
\dfrac{|\gamma'(t_0)|}{|\gamma(t_0)-t|^2}
\right)
\mu(dt).
$$
The later equality holds since
$$
4\sum\limits_{\gamma\in\Gamma}
\dfrac{|\gamma'(t_0)|}{|\gamma(t_0)-t|^2}
\ge
\dfrac{4}{|t_0-t|^2}\ge 1.
$$
Theorem follows.
\end{proof}
\begin{remark}
In what follows we assume that $m'_{t_0}$ is of bounded characteristic.
Since also measure corresponding to $m_{t_0}$ is pure point, the boundary values $m_{t_0}(t)$ exist
almost everywhere and are real. Therefore,
\eqref{211113_06'} holds for almost every $t\in\mathbb T$.
\end{remark}
\begin{theorem}\label{230311-01}
If $m'_{t_0}$ is of bounded characteristic, then %%% $m'_{t_0}$ is of Smirnov class and
\be\label{230530-05}
m'_{t_0}=\dfrac{\D_{t_0}}{\varphi},
\ee
where $\D_{t_0}$ is an inner function and $\varphi$ is an outer function, $|\varphi|\le 4$.
\end{theorem}
\begin{proof}
We enumerate elements of the group $\G$, $\G=\{\g_k\}$, and consider functions (for $0<r<1$)
$$
u_{n,r}(\z)=\sum\limits_{k=1}^n
\dfrac{|\gamma'_k(t_0)|}{|\gamma_k(t_0)-r\z|^2}
=
\sum\limits_{k=1}^n\overline{\phi_{k}(\z)}\phi_{k}(\z),
$$
where
$$
\phi_k(\z)=\dfrac{\sqrt{|\g'_k(t_0)|}}{\g_k(t_0)-r\z}.
$$
$\log u_{n,r}$ is a subharmonic function on $\bbD$. Indeed,
\begin{equation*}%\label{21oct05}
\frac{\partial^2}{\partial \z \partial \overline \z}\log u_{n,r}(\z)
= -\frac{1}{u_{n,r}^2(\z)}\frac{\partial u_{n,r}}{\partial \z}
\frac{\partial u_{n,r}}{\partial \overline \z}+\frac{1}{u_{n,r}}
\frac{\partial^2 u_{n,r}}{\partial \z \partial \overline \z}
\end{equation*}
$$
=\frac{1}{u_{n,r}^2(\z)}
\left\{
\sum\limits_{k=1}^n\overline{\phi_{k}(\z)}\phi_{k}(\z)
\sum\limits_{k=1}^n\overline{\phi'_{k}(\z)}\phi'_{k}(\z)
-
\sum\limits_{k=1}^n\overline{\phi_{k}(\z)}\phi'_{k}(\z)
\sum\limits_{k=1}^n\overline{\phi'_{k}(\z)}\phi_{k}(\z)
\right\},
$$
which is nonnegative by the Cauchy-Schwarz inequality. Since $\log u_{n,r}$ is also continuous on $\overline{\bbD}$, we have
$$
\int\limits_{\mathbb T}\log
\sum\limits_{k=1}^n
\frac{|\gamma'_k(t_0)|}{|\gamma_k(t_0)-rt|^2}
\frac{1-|\z|^2}{|t-\z|^2}\mu(dt)
\ge
\log
\sum\limits_{k=1}^n
\frac{|\gamma'_k(t_0)|}{|\gamma_k(t_0)-r\z|^2}
.
$$
By Lemma \ref{20230204-03},
$$
\frac{|\gamma'_k(t_0)|}{|\gamma_k(t_0)-rt|^2}
\le 4
\frac{|\gamma'_k(t_0)|}{|\gamma_k(t_0)-t|^2}.
$$
Therefore, Dominated Convergence Theorem can be used to pass to the limit when $r\nearrow 1$
$$
\int\limits_{\mathbb T}\log
\sum\limits_{k=1}^n
\frac{|\gamma'_k(t_0)|}{|\gamma_k(t_0)-t|^2}
\frac{1-|\z|^2}{|t-\z|^2}\mu(dt)
\ge
\log
\sum\limits_{k=1}^n
\frac{|\gamma'_k(t_0)|}{|\gamma_k(t_0)-\z|^2}
.
$$
Since
$$
\log
\sum\limits_{k=1}^n
\frac{|\gamma'_k(t_0)|}{|\gamma_k(t_0)-t|^2}
\ge \log\dfrac{1}{4},
$$
we can apply Monotone Convergence Theorem to send $n$ to $\infty$
$$
\int\limits_{\mathbb T}\log
\sum\limits_{\g\in\G}
\frac{|\gamma'(t_0)|}{|\gamma(t_0)-t|^2}
\frac{1-|\z|^2}{|t-\z|^2}\mu(dt)
\ge
\log
\sum\limits_{\g\in\G}
\frac{|\gamma'(t_0)|}{|\gamma(t_0)-\z|^2}
.
$$
The latter inequality yields, in view of Lemma \ref{17oct15},
\be\label{230312-02}
\int\limits_{\mathbb T}\log
|m'_{t_0}(t)|
\frac{1-|\z|^2}{|t-\z|^2}\mu(dt)
\ge
\log
\sum\limits_{\g\in\G}
\frac{|\gamma'(t_0)|}{|\gamma(t_0)-\z|^2}
\ge \log |m'_{t_0}(\z)|
.
\ee
Since $m'_{t_0}$ is of bounded characteristic, it can be written as
$$
m'_{t_0}=\dfrac{\D_1 O_1}{\D_2 O_2},
$$
where $\D_1, \D_2$ are inner functions, $O_1, O_2$ are bounded outer functions. Then \eqref{230312-02} reads as
$$
\log \left|\dfrac{O_1(\z)}{O_2(\z)}\right|\ge \log \left|\dfrac{\D_1(\z)O_1(\z)}{\D_2(\z)O_2(\z)}\right|,
$$
that is,
$$
\left|\dfrac{\D_1(\z)}{\D_2(\z)}\right|\le 1.
$$
This means that
$$
\D_{t_0}=\dfrac{\D_1}{\D_2}
$$
is an inner function. Since $|m'_{t_0}(t)|\ge 1/4$,
$$
\dfrac{O_1}{O_2}=\dfrac{1}{\varphi},
$$
$\varphi$ is outer, $|\varphi|\le 4$.
\end{proof}
\begin{corollary}\label{230514-02}
\be\label{230514-01}
\dfrac{|m'_{t_0}(\z)|}{|\D_{t_0}(\z)|}
\ge
\sum\limits_{\g\in\G}
\frac{|\gamma'(t_0)|}{|\gamma(t_0)-\z|^2}
\ge |m'_{t_0}(\z)|
.
\ee
\end{corollary}
\begin{proof}
This is an immediate consequence of \eqref{230312-02} since the most left term there is
$$
\int\limits_{\mathbb T}\log
|m'_{t_0}(t)|
\frac{1-|\z|^2}{|t-\z|^2}\mu(dt)
=
\int\limits_{\mathbb T}\log
\dfrac{1}{|\varphi(t)|}
\frac{1-|\z|^2}{|t-\z|^2}\mu(dt)
$$
\be\label{230530-01}
=
\log\dfrac{1}{|\varphi(\z)|}
=\log\dfrac{|m'_{t_0}(\z)|}{|\D_{t_0}(\z)|}.
\ee
\end{proof}

\section{Automorphic Properties of $m'_{t_0}$, $\varphi$ and $\D_{t_0}$.}
\begin{proposition}
\be\label{230529-01}
m'_{t_0}(\g(\z))\cdot\g'(\z)
=m'_{t_0}(\z).
\ee
\end{proposition}
\begin{proof}
By formula \eqref{221003_01},
$$
m'_{t_0}(\g(\z))\cdot\g'(\z)
=it_0\sum\limits_{\widetilde\gamma\in\Gamma}
\dfrac{\widetilde\gamma'(\g(\z))\cdot\g'(\z)}{(\widetilde\gamma(\g(\z))-t_0)^2}
=m'_{t_0}(\z).
$$
\end{proof}
\begin{lemma}\label{230530-08}
$$
\frac{1-|\g(\z)|^2}{|t-\g(\z)|^2}=
\frac{1-|\z|^2}{|\g^{-1}(t)-\z|^2}
\left|(\g^{-1})'(t)\right|.
$$
\end{lemma}
\begin{proof}
Let
$$
\g(\z)=\dfrac{a\z+b}{\overline b \z +\overline a},
$$
where $|a|^2-|b|^2=1$. Then
$$
\frac{1-|\g(\z)|^2}{|\g(\z)-t|^2}
=\frac{|\overline b\z+\overline a|^2-|a\z+b|^2}{|a\z+b-t(\overline b\z+\overline a)|^2}
$$
$$
=\dfrac{1-|\z|^2}{\left|\z-\dfrac{t\overline a-b}{-t\overline b+a}\right|^2}
\cdot
\dfrac{1}{|-t\overline b+a|^2}
$$
$$
=\frac{1-|\z|^2}{|\z-\g^{-1}(t)|^2}
\left|(\g^{-1})'(t)\right|,
$$
since
$$
\g^{-1}(t)=\dfrac{t\overline a-b}{-t\overline b+a}
\quad\text{and}\quad
(\g^{-1})'(t)=\dfrac{1}{(-t\overline b+a)^2}.
$$
\end{proof}
\begin{proposition}
\be\label{230530-02}
|\varphi(\g(\z))|
=|\varphi(\z)|\cdot|\g'(\z)|.
\ee
\end{proposition}
\begin{proof}
By \eqref{230530-01},
$$
\log|\varphi(\z)|
=
-\int\limits_{\mathbb T}\log
|m'_{t_0}(t)|
\frac{1-|\z|^2}{|t-\z|^2}\mu(dt).
$$
Then
$$
\log|\varphi(\g(\z))|
=
-\int\limits_{\mathbb T}\log
|m'_{t_0}(t)|
\frac{1-|\g(\z)|^2}{|t-\g(\z)|^2}\mu(dt)
$$
(by Lemma \ref{230530-08})
$$
=
-\int\limits_{\mathbb T}\log
|m'_{t_0}(t)|
\frac{1-|\z|^2}{|\g^{-1}(t)-\z|^2}
\left|(\g^{-1})'(t)\right|\mu(dt)
$$
(by substitution $t := \g(t)$)
$$
=
-\int\limits_{\mathbb T}\log
|m'_{t_0}(\g(t))|
\frac{1-|\z|^2}{|\g^{-1}(\g(t))-\z|^2}
\left|(\g^{-1})'(\g(t))\right||\g'(t)|\mu(dt)
$$
$$
=
-\int\limits_{\mathbb T}\log
|m'_{t_0}(\g(t))|
\frac{1-|\z|^2}{|t-\z|^2}\mu(dt).
$$
(by formula \eqref{230529-01})
$$
=
-\int\limits_{\mathbb T}\log
\dfrac{|m'_{t_0}(t)|}{|\g'(t)|}
\frac{1-|\z|^2}{|t-\z|^2}\mu(dt)
$$
$$
=
\log|\varphi(\z)|+\log|\g'(\z)|,
$$
since $\g'(\z)=\dfrac{1}{(\overline b\z+\overline a)^2}$ is an outer function.
Thus,
$$
|\varphi(\g(\z))|
=|\varphi(\z)|\cdot|\g'(\z)|.
$$
\end{proof}
\begin{proposition}
There exists a character of group $\G$ (denote it as $\d_{t_0}$) such that
\be\label{230530-03}
\D_{t_0}(\g(\z))=\d_{t_0}(\g)\D_{t_0}(\z).
\ee
\end{proposition}
\begin{proof}
Combining \eqref{230529-01} and \eqref{230530-02} we get
\be\label{230530-04}
|\varphi(\g(\z))\cdot m'_{t_0}(\g(\z))|
=|\varphi(\z)\cdot m'_{t_0}(\z)|.
\ee
In view of \eqref{230530-05}, this means that
$$
|\D_{t_0}(\g(\z))|=|\D_{t_0}(\z)|.
$$
The later yields \eqref{230530-03}.
\end{proof}

\section{Existence of $\a$-automorphic Functions in $1+(t-t_0)H^2$.}
%%{$H^2(\a)\cap \left(1+(t-t_0)H^2\right)\ne\emptyset$.}
\begin{definition}
Let $\G$ be a Fuchsian group and $\a$ be a unitary character of $\G$.
Poincar\'e theta series associated with Martin function $m_{t_0}$ is defined as
$$
(\mathcal P^\a f) (\z)
=\frac{\sum\limits_{\gamma\in\Gamma}\overline{\a(\gamma)}f(\g(\z))
\dfrac{\gamma'(\z)}{(\gamma(\z)-t_0)^2}}
{\sum\limits_{\gamma\in\Gamma}\dfrac{\gamma'(\z)}{(\gamma(\z)-t_0)^2}}
=\frac{\sum\limits_{\gamma\in\Gamma}\a(\gamma)f(\g^{-1}(\z))
\dfrac{\gamma'(t_0)}{(\gamma(t_0)-\z)^2}}
{\sum\limits_{\gamma\in\Gamma}\dfrac{\gamma'(t_0)}{(\gamma(t_0)-\z)^2}}
.
$$
See Appendix (Lemma \ref{230605-01}) for equivalence of the two forms. One can see from the first form that
$\mathcal P^\a f$ is $\a$ automorphic.
Observe that up to a constant factor the denominator is $m'_{t_0}(\z)$.
\end{definition}
\begin{theorem}\label{221231_04}
Assume that the following nontangential limits exist
$$
|\D_{t_0}(t_0)|=1,\quad \D'_{t_0}(t_0)\quad\text{\ is\ finite},
$$
that is,
\be\label{230312-03}
\D_{t_0}(\z)=\D_{t_0}(t_0)+ \D'_{t_0}(t_0)(\z-t_0)+o(\z-t_0),\quad |\D_{t_0}(t_0)|=1,
\ee
as $\z$ approaches $t_0$ nontangentially.
Assume that
\be\label{221231_01}
|t_0-\z|\sum\limits_{\gamma\ne e}
\dfrac{|\gamma'(t_0)|}{|\gamma(t_0)-\zeta|^2}=o(1)
\ee
as $\z\to t_0$ nontangentially
\footnote{By the Dominated Convergence Theorem, it is true that
$
|t_0-\z|^2\sum\limits_{\gamma\ne e}
\dfrac{|\gamma'(t_0)|}{|\gamma(t_0)-\zeta|^2}=o(1)
$, but here we need more.
}.
%%%\be\label{221231_01}
%%%\sum\limits_{\gamma\ne e}
%%%\dfrac{|\gamma'(t_0)|}{|\gamma(t_0)-t_0|^2}<\infty.
%%%\ee
Then for every character $\alpha$
\be\label{230517-01}
\mathcal P^\alpha \Delta_{t_0}\in H^\infty,\quad \Vert\mathcal P^\alpha \Delta_{t_0}\Vert_{H^\infty}\le 1,
\ee
\be\label{230517-02}
(\mathcal P^\alpha \Delta_{t_0}) (\z)=\D_{t_0}(t_0)+\D'_{t_0}(t_0)(\z-t_0)+o(\z-t_0)
\ee
as $\z$ approaches $t_0$ nontangentially, and %%%, therefore, by Carath\' eodory - Julia theorem,
\be\label{230517-03}
\int\limits_{\bbT}\dfrac{|(\mathcal P^\alpha \Delta_{t_0})(t)-\D_{t_0}(t_0)|^2}{|t-t_0|^2}\mu(dt)
+
\int\limits_{\bbT}\dfrac{1-|(\mathcal P^\alpha \Delta_{t_0})(t)|^2}{|t-t_0|^2}\mu(dt)
=t_0\dfrac{\D_{t_0}'(t_0)}{\D_{t_0}(t_0)}.
\ee
\end{theorem}
\begin{proof}
$$
(\mathcal P^\alpha \Delta_{t_0}) (\z)
=\frac{\sum\limits_{\gamma\in\Gamma}\alpha(\gamma)\D_{t_0}(\g^{-1}(\z))
\dfrac{\gamma'(t_0)}{(\gamma(t_0)-\z)^2}}
{\sum\limits_{\gamma\in\Gamma}\dfrac{\gamma'(t_0)}{(\gamma(t_0)-\z)^2}}
=\D_{t_0}(\z)\frac{\sum\limits_{\gamma\in\Gamma}\alpha(\gamma)\overline{\d_{t_0}(\gamma)}
\dfrac{\gamma'(t_0)}{(\gamma(t_0)-\z)^2}}
{\sum\limits_{\gamma\in\Gamma}\dfrac{\gamma'(t_0)}{(\gamma(t_0)-\z)^2}}
$$
Hence $\mathcal P^\alpha \Delta_{t_0}$ is $\b$ automorphic and
$$
|(\mathcal P^\alpha \Delta_{t_0}) (\z)|\le
|\D_{t_0}(\z)|\dfrac{\sum\limits_{\gamma\in\Gamma}
\dfrac{|\gamma'(t_0)|}{|\gamma(t_0)-\z|^2}}
{\left|\sum\limits_{\gamma\in\Gamma}\dfrac{\gamma'(t_0)}{(\gamma(t_0)-\z)^2}\right|}
=
\dfrac{|\D_{t_0}(\z)|}{|m'_{t_0}(\z)|}\sum\limits_{\gamma\in\Gamma}
\dfrac{|\gamma'(t_0)|}{|\gamma(t_0)-\z|^2}\le 1.
$$
The later inequality is due to Corollary \ref{230514-02}.
Thus, $\mathcal P^\alpha \Delta_{t_0}\in H^\infty(\b)$, $\Vert\mathcal P^\alpha \Delta_{t_0}\Vert_\infty\le 1$.
On the other hand
$$
(\mathcal P^\alpha \Delta_{t_0}) (\z)=\D_{t_0}(\z)\frac{1+(t_0-\z)^2\sum\limits_{\gamma\ne e}\alpha(\gamma)\d_{t_0}(\gamma)
\dfrac{\gamma'(t_0)}{(\gamma(t_0)-\z)^2}}
{1+(t_0-\z)^2\sum\limits_{\gamma\ne e}\dfrac{\gamma'(t_0)}{(\gamma(t_0)-\z)^2}}
$$
due to assumption \eqref{221231_01}
$$
=\D_{t_0}(\z)\frac{1+o(t_0-\z)}
{1+o(t_0-\z)}
$$
$$
=\D_{t_0}(\z)(1+o(\z-t_0))=\D_{t_0}(t_0)+\D'_{t_0}(t_0)(\z-t_0)+o(\z-t_0).
$$
Therefore, $(\mathcal P^\alpha \Delta_{t_0}) (\z)$ has nontangential boundary value $\D_{t_0}(t_0)$ and
$$
\dfrac{(\mathcal P^\alpha \Delta_{t_0}) (\z)-\D_{t_0}(t_0)}{\z-t_0}
$$
has nontangential boundary value $\D'_{t_0}(t_0)$. Hence,
by Carath\' eodory - Julia theorem (Theorem \ref{ACJ-03sep01} in Appendix),
$$
\int\limits_{\bbT}\dfrac{|(\mathcal P^\alpha \Delta_{t_0})(t)-\D_{t_0}(t_0)|^2}{|t-t_0|^2}\mu(dt)
+
\int\limits_{\bbT}\dfrac{1-|(\mathcal P^\alpha \Delta_{t_0})(t)|^2}{|t-t_0|^2}\mu(dt)
=t_0\dfrac{\D_{t_0}'(t_0)}{\D_{t_0}(t_0)}.
$$
\end{proof}
\begin{corollary}\label{230517-04}
Under assumptions of Theorem \ref{221231_04},
for every character $\a$ there exists an
$\a$ automorphic function $g$ such that
\be\label{211115_13}
g\in 1+(t-t_0)H^2.
\ee
\end{corollary}
\begin{proof}
Indeed, formula \eqref{230517-03} yields, in particular, that for every character $\a$
$$
\dfrac{(\mathcal P^\alpha \Delta_{t_0})(t)-\D_{t_0}(t_0)}{t-t_0}\in H^2.
$$
Hence
$$
\dfrac{(\mathcal P^\alpha \Delta_{t_0})(t)\overline{\D_{t_0}(t_0)}-1}{t-t_0}\in H^2.
$$
That is
$
g(t)=(\mathcal P^\alpha \Delta_{t_0})(t)\overline{\D_{t_0}(t_0)}
$
is the requisite function.
\end{proof}
\begin{remark}\label{230623-01}
Let $|t_0|=1$. Then
for every choice of numbers $w_0$, $w'_0$ such that $|w_0|=1$,
\be\label{230623-04}
t_0\dfrac{w'_0}{w_0}\ge
t_0\dfrac{\D_{t_0}'(t_0)}{\D_{t_0}(t_0)}
\ee
and for every character $\a$ there exists an $\a$ automorphic analytic on $\bbD$ function $w(\z)$,
$|w(\z)|\le 1$, such that
\be\label{230623-05}
w(\z)=w_0+w'_0(\z-t_0)+o(\z-t_0)
\ee
as $\z$ approaches $t_0$ nontangentially.
\end{remark}
\begin{proof}
Let $f(\z)$ be a function analytic on $\bbD$, $|f(\z)|\le 1$.  Assume that
\be\label{230623-02}
f(\z)=f_0+f'_0(\z-t_0)+o(\z-t_0)
\ee
as $\z$ approaches $t_0$ nontangentially. Then for every character $\a$
\be\label{230623-03}
w(\z):=
(\mathcal P^\alpha \Delta_{t_0}f) (\z)=\D_{t_0}(t_0)f_0+(\D'_{t_0}(t_0)f_0+\D_{t_0}f'_0)(\z-t_0)+o(\z-t_0).
\ee
Hence,
$$
w_0=\D_{t_0}(t_0)f_0, \quad w'_0=\D'_{t_0}(t_0)f_0+\D_{t_0}f'_0
$$
and
\be\label{230623-07}
t_0\dfrac{w'_0}{w_0}
=t_0\dfrac{\D_{t_0}'(t_0)}{\D_{t_0}(t_0)}
+
t_0\dfrac{f'_0}{f_0}.
\ee
Also $w(\z)$ is analytic and bounded in modulus by $1$. It remains to select an $f$, $|f(\z)\le 1$, for given $w_0$ and $w'_0$.
It is possible due to \eqref{230623-04}.
\end{proof}
One can use this normalization $\D_{t_0}(t_0)=1$. In this case property \eqref{230312-03} reads as
\be\label{230530-07}
\D_{t_0}(\z)=1+ \D'_{t_0}(t_0)(\z-t_0)+o(\z-t_0),\quad |\D_{t_0}(t_0)|=1,
\ee
\begin{remark}
Properties \eqref{230530-07} %%% $($with normalization $\D_{t_0}(t_0)=1)$
and \eqref{221231_01} imply that
\be\label{230530-06}
\dfrac{\varphi(\z)}{(\z-t_0)^2}= 1+\D'_{t_0}(t_0)(\z-t_0)+o(\z-t_0),
\ee
as $\z$ approaches $t_0$ nontangentially.
\end{remark}
\begin{proof}
Since
$$
\varphi m'_{t_0}=\D_{t_0},
$$
we have
$$
\dfrac{\varphi(\z)}{(\z-t_0)^2}
\left(
1+(\z-t_0)^2
it_0\sum\limits_{\gamma\ne e}\dfrac{\gamma'(t_0)}{(\gamma(t_0)-\zeta)^2}
\right)
=1+\D'_{t_0}(t_0)(\z-t_0)+o(\z-t_0).
$$
Using \eqref{221231_01}, we get
$$
\dfrac{\varphi(\z)}{(\z-t_0)^2}
(1+o(\z-t_0))
=1+\D'_{t_0}(t_0)(\z-t_0)+o(\z-t_0).
$$
From here we get \eqref{230530-06}.
\end{proof}
The author does not have a necessary condition of the property that
for every character $\a$ there exist
$\a$ automorphic functions in $1+(t-t_0)H^2$.

\section{Kernels $k^\a_{t_0}$ and Their Properties.}\label{230606-06}
\begin{theorem}\label{220105_01}
Under assumptions \eqref{230312-03} and \eqref{221231_01} for every character $\a$ the minimum
of $\ \left\Vert\dfrac{g(t)-1}{t-t_0}\right\Vert_{H^2}$ over $\a$ automorphic $g$ is uniquely attained.
\end{theorem}
\begin{proof}
By Theorem \ref{221231_04}, there exist  $\a$ automorphic %% Smirnov class
functions $g$ that satisfy
\be\label{211115_03}
\dfrac{g-1}{t-t_0}=h\in H^2.
\ee
Equivalently
\be\label{230312-04}
g=1+(t-t_0)h,\quad h\in H^2.
\ee
%%%  then $g$ is also in $H^2$ and, therefore, it is automatically in Smirnov class.
The set of functions $h$ in $H^2$, such that the corresponding function $g$ defined by \eqref{230312-04} is $\a$
automorphic, is convex and closed. The set is nonempty as stated above.
The assertion now is a standard Hilbert space property applied to the distance from $0$ to this convex set.
\end{proof}
\begin{definition}
We denote by $k^\a_{t_0}$ the unique function that gives {\it minimum} to the integral
\be\label{210916_03}
\int\limits_{\bbT}\dfrac{|g(t)-1|^2}{|t-t_0|^2}\mu(dt),\quad
\dfrac{g(t)-1}{t-t_0}\in H^2,
\ee
$g$ is $\a$ automorphic.
\end{definition}

%%%%%In what follows we will use these notations: for $1\le p <\infty$ $\cL^p$ that is the of functions $f$ such that
%%%%%\be\label{211113_10}
%%%%%\int\limits_{\bbT}|f|^p \dfrac{\mu(dt)}{|t-t_0|^2}<\infty.
%%%%%\ee
%%%%%That is
%%%%%\be\label{211113_11}
%%%%%\dfrac{f^p}{(t-t_0)^2}\in L^1.
%%%%%\ee
%%%%%We define $\cH^p$ as the space of functions $f$ such that
%%%%%\be\label{211113_12}
%%%%%\dfrac{f^p}{(t-t_0)^2}\in H^1.
%%%%%\ee
%%%%%Observe that $\cH^p$ is a subset of $H^p$. Moreover, it is dense in $H^p$.

\begin{theorem}\label{211115_01}
$k^\a_{t_0}$ is uniquely determined by these conditions: $g$ is
$\a$ automorphic,
\be\label{210916_01'}
\dfrac{g-1}{t-t_0}\in H^2.
\ee
and
\be\label{210916_02}
\int\limits_{\bbT}\dfrac{(\overline{g}-1)f}{|t-t_0|^2}\mu(dt)=0
\ee
for every $\a$ automorphic $f$ such that $\dfrac{f}{t-t_0}\in H^2$.
\end{theorem}
\begin{proof}
Check first that $k^\a_{t_0}$ satisfies \eqref{210916_02}. Assume for the sake of
contradiction that \eqref{210916_02} fails. That is,
that there exists $\a$ automorphic $f$ such that $\dfrac{f}{t-t_0}\in H^2$ and
\be\label{210916_04}
\int\limits_{\bbT}\dfrac{(\overline{k^\a_{t_0}}-1)f}{|t-t_0|^2}\mu(dt)>0.
\ee
Let $c>0$. Consider
\be\label{210916_05}
\int\limits_{\bbT}\dfrac{|k^\a_{t_0}-cf-1|^2}{|t-t_0|^2}\mu(dt)
=\int\limits_{\bbT}\dfrac{|k^\a_{t_0}-1|^2-c(\overline{k^\a_{t_0}}-1)f-c({k^\a_{t_0}}-1)\overline f+c^2|f|^2}
{|t-t_0|^2}\mu(dt).
\ee
By choosing sufficiently small $c$, one can see that $k^\a_{t_0}$ is not an extremum. Contradiction. Therefore,
\eqref{210916_02} holds.

We show now that $\a$ automorphic function $g$ that satisfies \eqref{210916_01'} and \eqref{210916_02} is unique. Indeed, if $g_1$ and $g_2$ are such functions, then their difference is $\a$ automorphic and
$$
\dfrac{g_1-g_2}{t-t_0}\in H^2.
$$
Therefore,
$$
\int\limits_{\bbT}\dfrac{(\overline{g_1}-1)(g_1-g_2)}{|t-t_0|^2}\mu(dt)=0
$$
and
$$
\int\limits_{\bbT}\dfrac{(\overline{g_2}-1)(g_1-g_2)}{|t-t_0|^2}\mu(dt)=0.
$$
This implies
$$
\int\limits_{\bbT}\dfrac{(\overline{g_1}-\overline{g_2})(g_1-g_2)}{|t-t_0|^2}\mu(dt)=0.
$$
Theorem follows.
\end{proof}

\section{Automorphic Carath\' eodory-Julia: Necessary Condition.}

\begin{theorem}\label{230606-07}
Let $w$ be a $\b$-automorphic analytic function on the unit disk, $|w(\zeta)|\le 1$. Let $t_0$ be a point on the unit circle, $|t_0|=1$.
Assume that $w$ and $w'$ have
nontangential boundary values $w_0$, $|w_0|=1$ and $w'_0$, respectively, at this point $t_0$. Then for every character $\a$
\be\label{210915_02}
t_0\dfrac{w'_0}{w_0}
\ge
\int\limits_{\bbT}\dfrac{|k^{\a\b}_{t_0}-1|^2}{|t-t_0|^2}\mu(dt)
-\int\limits_{\bbT}\dfrac{|k^\a_{t_0}-1|^2}{|t-t_0|^2}\mu(dt).
\ee
\end{theorem}

\begin{proof}
$$
\int\limits_{\bbT}
\begin{bmatrix} \overline{k^{\a\b}_{t_0}} & -\overline{k^\a_{t_0}}w(t_0)  \end{bmatrix}
\begin{bmatrix} 1 & w \\ \overline w & 1\end{bmatrix}
\begin{bmatrix} k^{\a\b}_{t_0} \\ \\ -\overline{w(t_0)}k^\a_{t_0} \end{bmatrix}
\frac{1}{|t-t_0|^2}\mu(dt)
$$
$$ %%%%%%%%%%%%% \be\label{210819_02}
=\int\limits_{\bbT}
\dfrac{
|k^{\a\b}_{t_0}|^2-\overline{k^\a_{t_0}}w(t_0)\overline w k^{\a\b}_{t_0}
- \overline{k^{\a\b}_{t_0}} w\overline{w(t_0)}k^\a_{t_0}
+|k^\a_{t_0}|^2}
{|t-t_0|^2}\mu(dt)
$$ %%%%%%%%%%%%% \ee
$$
=\int\limits_{\bbT}
\dfrac{
(\overline{k^{\a\b}_{t_0}}-w(t_0)\overline w \overline{k^\a_{t_0}}) k^{\a\b}_{t_0}
+ \overline{k^{\a\b}_{t_0}}(k^{\a\b}_{t_0} - w\overline{w(t_0)}k^\a_{t_0})
+|k^\a_{t_0}|^2-|k^{\a\b}_{t_0}|^2}
{|t-t_0|^2}\mu(dt)
$$
since
$$
k^{\a\b}_{t_0} - w\overline{w(t_0)}k^\a_{t_0}
=  k^{\a\b}_{t_0}-1+1 -k^\a_{t_0} +(1 - w\overline{w(t_0)})k^\a_{t_0}
$$
$$
=  (k^{\a\b}_{t_0}-1)+(1 -k^\a_{t_0}) +(1 - w\overline{w(t_0)})(k^\a_{t_0}-1)
+(1 - w\overline{w(t_0)})
$$
is $\a\b$ automorphic and after dividing by $t-t_0$ is in $H^2$ (in the third term we use boundedness of $w$)
and in view of Theorem \ref{211115_01}
$$
=\int\limits_{\bbT}
\dfrac{
(\overline{k^{\a\b}_{t_0}}-w(t_0)\overline w \overline{k^\a_{t_0}})
+ (k^{\a\b}_{t_0} - w\overline{w(t_0)}k^\a_{t_0})
+|k^\a_{t_0}|^2-|k^{\a\b}_{t_0}|^2}
{|t-t_0|^2}\mu(dt)
$$
$$
=\int\limits_{\bbT}
\dfrac{
(\overline{k^{\a}_{t_0}}-w(t_0)\overline w \overline{k^\a_{t_0}})
+ (k^{\a}_{t_0} - w\overline{w(t_0)}k^\a_{t_0})
+|k^\a_{t_0}|^2-k^{\a}_{t_0}-\overline{k^{\a}_{t_0}}+1
-1-|k^{\a\b}_{t_0}|^2+k^{\a\b}_{t_0}+\overline{k^{\a\b}_{t_0}}}
{|t-t_0|^2}\mu(dt)
$$
since
$$
\dfrac{(1 - w\overline{w(t_0)})(k^\a_{t_0}-1)}
{(t-t_0)^2}\in H^1,
$$
$$
=\int\limits_{\bbT}
\dfrac{
(1-w(t_0)\overline w )
+ (1 - w\overline{w(t_0)})
+|k^\a_{t_0}-1|^2
-|k^{\a\b}_{t_0}-1|^2}
{|t-t_0|^2}\mu(dt).
$$
Thus,
$$
\int\limits_{\bbT}
\begin{bmatrix} \overline{k^{\a\b}_{t_0}} & -w(t_0)\overline{k^{\a}_{t_0}}  \end{bmatrix}
\begin{bmatrix} 1 & w \\ \overline w & 1\end{bmatrix}
\begin{bmatrix} k^{\a\b}_{t_0} \\ \\ -\overline{w(t_0)}k^{\a}_{t_0} \end{bmatrix}
\frac{1}{|t-t_0|^2} \mu(dt)
$$
\be\label{210915_01}
=t_0\dfrac{w'(t_0)}{w(t_0)}
+\int\limits_{\bbT}\dfrac{|k^\a_{t_0}-1|^2}{|t-t_0|^2}\mu(dt)
-\int\limits_{\bbT}\dfrac{|k^{\a\b}_{t_0}-1|^2}{|t-t_0|^2}\mu(dt).
\ee
Since the left hand side is nonnegative, we get
\be\label{210915_02'}
t_0\dfrac{w'(t_0)}{w(t_0)}
\ge
\int\limits_{\bbT}\dfrac{|k^{\a\b}_{t_0}-1|^2}{|t-t_0|^2}\mu(dt)
-\int\limits_{\bbT}\dfrac{|k^\a_{t_0}-1|^2}{|t-t_0|^2}\mu(dt).
\ee
\end{proof}
\begin{remark}
The author hopes that the converse is also true. That is, if $|t_0|=1$ and if numbers $w_0$, $|w_0|=1$ and $w'_0$
satisfy condition \eqref{210915_02} for every character $\a$, then there exists a
$\b$-automorphic analytic on the unit disk function $w(\z)$, $|w(\zeta)|\le 1$ such that $w$ and $w'$ have
nontangential boundary values $w_0$ and $w'_0$, respectively, at this point $t_0$. However, the author does
not have a proof of that. Probably some additional assumptions about group $\G$ are needed.
\end{remark}

\section{Additional Considerations.}

Inequality \eqref{210915_02} can be stated as
\be\label{210915_02''}
t_0\dfrac{w'_0}{w_0}
\ge
\underset{\a\in\G^*}{\sup}
\left(
\int\limits_{\bbT}\dfrac{|k^{\a\b}_{t_0}-1|^2}{|t-t_0|^2}\mu(dt)
-\int\limits_{\bbT}\dfrac{|k^\a_{t_0}-1|^2}{|t-t_0|^2}\mu(dt)
\right).
\ee
Assume now that for every character $\b$ there exists a
$\b$-automorphic analytic on the unit disk function $w$, $|w(\zeta)|\le 1$ such that $w$ and $w'$ have
nontangential boundary values $w_0$ and $w'_0$, respectively, at a point $t_0$, $|t_0|=1$. Then
\be\label{210915_02'''}
t_0\dfrac{w'_0}{w_0}
\ge
\underset{\b\in\G^*}{\sup}
\underset{\a\in\G^*}{\sup}
\left(
\int\limits_{\bbT}\dfrac{|k^{\a\b}_{t_0}-1|^2}{|t-t_0|^2}\mu(dt)
-\int\limits_{\bbT}\dfrac{|k^\a_{t_0}-1|^2}{|t-t_0|^2}\mu(dt)
\right)
\ee
$$
=
\underset{\a\in\G^*}{\sup}
\int\limits_{\bbT}\dfrac{|k^{\a}_{t_0}-1|^2}{|t-t_0|^2}\mu(dt)
-
\underset{\a\in\G^*}{\inf}
\int\limits_{\bbT}\dfrac{|k^\a_{t_0}-1|^2}{|t-t_0|^2}\mu(dt)
$$
$$
=
\underset{\a\in\G^*}{\sup}
\int\limits_{\bbT}\dfrac{|k^{\a}_{t_0}-1|^2}{|t-t_0|^2}\mu(dt).
$$
The $\underset{\a\in\G^*}{\inf}$ is equal to $0$ since $k^\a_{t_0}=1$
for the identity character $\a$. Our goal here is to esimate the later $\underset{\a\in\G^*}{\sup}$.
\begin{theorem}
For every character $\a$
\be\label{230517-05}
%%\underset{\a\in\G^*}{\sup}
\int\limits_{\bbT}\dfrac{|k^{\a}_{t_0}-1|^2}{|t-t_0|^2}\mu(dt)
\le
t_0\dfrac{\D_{t_0}'(t_0)}{\D_{t_0}(t_0)}.
\ee
\end{theorem}
\begin{proof}
It was observed in Corollary \ref{230517-04} that formula \eqref{230517-03}
yields
$$
\dfrac{(\mathcal P^\alpha \Delta_{t_0})(t)\overline{\D_{t_0}(t_0)}-1}{t-t_0}\in H^2.
$$
Moreover, it follows from formula \eqref{230517-03} that
$$
\int\limits_{\bbT}\dfrac{\left|\cP^\a\left(\D_{t_0}\overline{\D_{t_0}(t_0)}\right)-1\right|^2}{|t-t_0|^2}\mu(dt)
\le t_0\dfrac{\D_{t_0}'(t_0)}{\D_{t_0}(t_0)}.
$$
Then, by the extremal property of $k_{t_0}^\a$,
\be\label{221231_05}
\int\limits_{\bbT}\dfrac{|k_{t_0}^\a-1|^2}{|t-t_0|^2}\mu(dt)
\le
\int\limits_{\bbT}\dfrac{\left|\cP^\a\left(\D_{t_0}\overline{\D_{t_0}(t_0)}\right)-1\right|^2}{|t-t_0|^2}\mu(dt)
\le t_0\dfrac{\D_{t_0}'(t_0)}{\D_{t_0}(t_0)}.
\ee
\end{proof}
To guarantee that bound \eqref{230517-05} is attained for some character $\a$ one needs an additional assumption on the group $\G$
called DCT (Direct Cauchy Theorem) property. This concept apparently was introduced in \cite{Hasu}.
It was extensively used and discussed in \cite{SY} and subsequent works.
The property is related to the validity of the Cauchy Integral Theorem/Formula when integrating
functions of a certain class over the boundary of the surface $\bbD/\G$.
We do not give here a detailed motivation of the property. We just state it in the form we need it.
%%% We also compare it to the form used in \cite{SY}.
\begin{definition}
We will say that group $\G$ has $t_0$ - DCT property if
\be\label{230623-06}
\int\limits_{\mathbb T}\overline{\left({\D_{t_0}(t)}\overline{\D_{t_0}(t_0)}-1\right)}f(t)\dfrac{\mu(dt)}{|t-t_0|^2} =0
\ee
for every $\d_{t_0}$ automorphic $f$ such that $\dfrac{f}{t-t_0}\in H^2$, where $\d_{t_0}$ is the character of $\D_{t_0}$.
\end{definition}
In view of Theorem \ref{211115_01}, this DCT property can be read as
$$
k^{\d_{t_0}}_{t_0}=\D_{t_0}(t)\overline{\D_{t_0}(t_0)}.
$$
The later implies that bound \eqref{230517-05} is attained for character $\d_{t_0}$ and
\be\label{230623-11}
\underset{\a\in\G^*}{\sup}
\int\limits_{\bbT}\dfrac{|k^{\a}_{t_0}-1|^2}{|t-t_0|^2}\mu(dt)
=
t_0\dfrac{\D_{t_0}'(t_0)}{\D_{t_0}(t_0)}.
\ee
Now Remark \ref{230623-01} can be completed as follows
\begin{remark}\label{230623-08}
Let $|t_0|=1$. We assume that assumptions \eqref{230312-03}, \eqref{221231_04} of Theorem \ref{221231_01}
and assumption \eqref{230623-06} hold true. Then
for every character $\a$ there exists an $\a$ automorphic analytic on $\bbD$ function $w(\z)$,
$|w(\z)|\le 1$, such that
\be\label{230623-10}
w(\z)=w_0+w'_0(\z-t_0)+o(\z-t_0)
\ee
as $\z$ approaches $t_0$ nontangentially {\bf if and only if}
\be\label{230623-09}
t_0\dfrac{w'_0}{w_0}\ge
t_0\dfrac{\D_{t_0}'(t_0)}{\D_{t_0}(t_0)}.
\ee
\end{remark}
\begin{proof}
Necessity of \eqref{230623-09} follows from \eqref{210915_02'''} and
\eqref{230623-11}. Sufficiency was proved in Remark \ref{230623-01}.
\end{proof}
%%%%\begin{conjecture}
%%%%Under the DCT assumption,
%%%%$$
%%%%\int\limits_{\bbT}\dfrac{|k^{\a}_{t_0}-1|^2}{|t-t_0|^2}\mu(dt)
%%%%$$
%%%%is a continuous function of $\a$.
%%%%\end{conjecture}

\section{Appendix}
\begin{lemma}\label{230605-01}
The derivative of the Martin function also equals
$$
m'_{t_0}(\z)=it_0\sum\limits_{\gamma\in\Gamma}
\dfrac{\gamma'(\z)}{(\gamma(\z)-t_0)^2}.
$$
\end{lemma}
\begin{proof}
Let,
$$
\g(\z)=\dfrac{a\z+b}{\overline b \z +\overline a},
$$
where $|a|^2-|b|^2=1$. Then
\be\label{20230202_01}
\g'(\z)=\dfrac{1}{(\overline b \z +\overline a)^2}
\ee
and
$$
\dfrac{\gamma'(\z)}{(\gamma(\z)-t_0)^2}
$$
$$
=\dfrac{1}{(\overline b \z +\overline a)^2\left[\dfrac{a\z+b}{\overline b \z +\overline a}-t_0\right]^2}
=\dfrac{1}{[(a\z+b)-t_0(\overline b \z+\overline a)]^2}
$$
$$
=\dfrac{1}{[(-\overline b t_0+a)\z-(\overline a t_0-b)]^2}
=\dfrac{1}{(-\overline b t_0+a)^2\left[\z-\dfrac{\ \overline a t_0-b}{-\overline b t_0+a}\right]^2}
$$
$$
=\dfrac{(\gamma^{-1})'(t_0)}{(\z-\gamma^{-1}(t_0))^2},
$$
since
$$
\g^{-1}(\z)=\dfrac{\ \overline a \z-b}{-\overline b \z +a}
$$
(the inverse matrix).
\end{proof}
\begin{theorem}[Carath\'eodory--Julia, \cite{Car}]\label{ACJ-03sep01}
Let function $w$ be analytic in the unit disk and bounded in modulus by $1$.
Let $t$ be a point on the unit circle. The following are equivalent:
\begin{eqnarray}
&&(1) \quad d_1:={\displaystyle\liminf_{\zeta\to
t}\frac{1-|w(\zeta)|^2}{1-|\zeta|^2}}<\infty\quad (|\zeta|<1, \zeta {\text\ approaches\ } t
{\text\ in\ an\ arbitrary\ way });\nonumber \\
&&(2) \quad
d_2:={\displaystyle\lim_{\zeta\to t}
\frac{1-|w(\zeta)|^2}{1-|\zeta|^2}}<\infty
\quad (\zeta {\text\ approaches\ } t
{\text\ nontangentially });\nonumber\\
&&(3)\quad \mbox{Finite nontangential limits} \; \;  \nonumber\\
&&\qquad\qquad w(t):={\displaystyle\lim_{\zeta\to
t}w(\zeta)} \; \; \mbox{and} \; \;
d_3:={\displaystyle\lim_{\zeta\to
t}\frac{1-w(\zeta)\overline{w(t)}}{1-\zeta\bar{t}}}\nonumber\\
&&\qquad\ \mbox{exist},\ |w(t)|=1.\nonumber\\
&&(4)\quad \mbox{Finite nontangential limits} \; \;  \nonumber\\
&&\qquad\qquad w(t):={\displaystyle\lim_{\zeta\to
t}w(\zeta)} \; \; \mbox{and} \; \;
w'(t)={\displaystyle\lim_{\zeta\to
t}\frac{w(\zeta)-w(t)}{\zeta-t}}\nonumber\\
&&\qquad\ \mbox{exist}, |w(t)|=1.\ w'(t)\ \text{is called the angular derivative at}\ t.\nonumber\\
&&(5)\quad \mbox{Finite nontangential limits} \; \; \nonumber\\
&&\qquad\qquad w(t):={\displaystyle\lim_{\zeta\to
t}w(\zeta)} \; \; \mbox{and} \; \; w'_0:={\displaystyle\lim_{\zeta\to
t}w^\prime(\zeta)}\nonumber\\
&&\qquad\ \mbox{exist},\ |w(t)|=1.\nonumber\\
&&(6)\quad \mbox{There exist a constant $w_0$, $|w_0|=1$ and a constant $d\ge 0$
} \; \; \nonumber\\
&&\qquad \mbox{\ such that the boundary Schwarz-Pick inequality holds}\nonumber\\
&&\qquad\qquad
\left|
\frac{w(\zeta)-w_0}{\zeta-t}
\right|^2\le d\cdot \frac{1-|w(\zeta)|^2}{1-|\zeta|^2},\quad |\zeta|<1;
\label{ACJ-28oct01}\\
&&\qquad \mbox{inequality \eqref{ACJ-28oct01} implies that the following nontangential limit} \; \; \nonumber\\
&&\qquad\qquad w(t):={\displaystyle\lim_{\zeta\to
t}w(\zeta)} \; \; \mbox{exists\ and} \; \; w(t)=w_0;\nonumber\\
&&\qquad \mbox{\ we denote the smallest constant $d$ that works for \eqref{ACJ-28oct01}  by $d_4$.}\nonumber
\end{eqnarray}
When these conditions hold, we have $w'_0=w'(t)$ and
$$d_1=d_2=d_3=d_4=t\frac{w'(t)}{w(t)}=|w'(t)|.$$
This number is equal to $0$ if and only if $w$ is a unimodular constant.
\end{theorem}
Under conditions of the Carath\'eodory-Julia Theorem the following formula holds
\be\label{230525-01}
\int\limits_{\bbT}\dfrac{|w(\tau)-w(t)|^2}{|\tau -t|^2}\mu(d\tau)
+
\int\limits_{\bbT}\dfrac{1-|w(\tau)|^2}{|\tau-t|^2}\mu(d\tau)
=t\dfrac{w'(t)}{w(t)}.
\ee
In particular,
\be\label{230606-01}
\dfrac{w(\cdot)-w(t)}{\cdot -t}\in H^2.
\ee
For a modern exposition of the Carath\'eodory-Julia Theorem see, e. g., \cite{BoKh} and further references there, also see
\cite{Kh-Yud-2019}.
\begin{theorem}[Frostman, \cite{Fro}]\label{ACJ-19oct01}
Let $w$, $w_n$ be analytic functions bounded in modulus by $1$ on the unit disk.
Assume that $|w_{n+1}(\zeta)|\le|w_n(\zeta)|$ for every $|\zeta|<1$ and every $n$.
Assume that $w_n(\zeta)$ converges to $w(\zeta)$ for every $|\zeta|<1$.
Let $|t|=1$. Assume that nontangential boundary values $w_n(t)$,
$w'_n(t)$ exist and that $|w_n(t)|=1$, $w'_n(t)$
are finite. Then nontangential boundary values $w(t)$,
$w'(t)$ exist and  $|w(t)|=1$, $w'(t)$ is finite if and only if
$|w_n'(t)|$ are bounded above. In this case
$$
w(t)=\lim w_n(t),\quad w'(t)=\lim w_n'(t).
$$
\end{theorem}
\begin{proof}
Observe first that the sequence $|w_n'(t)|$ is increasing.
Indeed, by assumption,
$$
\frac{1-|w_n(\zeta)|^2}{1-|\zeta|^2}\le \frac{1-|w_{n+1}(\zeta)|^2}{1-|\zeta|^2}.
$$
Therefore,
$$
\lim_{\zeta\to t} \frac{1-|w_n(\zeta)|^2}{1-|\zeta|^2}\le
\lim_{\zeta\to t} \frac{1-|w(\zeta)|^2}{1-|\zeta|^2}.
$$
That is, in view of Theorem \ref{ACJ-03sep01},
$$
|w'_n(t)|\le |w'(t)|.
$$
If $|w_n'(t)|$ are bounded above, then a finite limit exists
$$
d=\lim |w_n'(t)|.
$$
We have boundary Schwarz - Pick inequalities \eqref{ACJ-28oct01}
$$
\left|
\frac{w_n(\zeta)-w_n(t)}{\zeta-t}
\right|^2\le |w_n'(t)|\cdot \frac{1-|w_n(\zeta)|^2}{1-|\zeta|^2},\quad |\zeta|<1.
$$
Let $w_0$ be any subsequential limit of $w_n(t)$ ($|w_0|=1$), then
\be\label{ACJ-2019-11-03-01}
\left|
\frac{w(\zeta)-w_0}{\zeta-t}
\right|^2\le d\cdot \frac{1-|w(\zeta)|^2}{1-|\zeta|^2},\quad |\zeta|<1.
\ee
From here we conclude that nontangential limit $w(t)$ of $w$ exists and equals $w_0$.
On the other hand this means that $w_n(t)$ has only one subsequential limit $w(t)$, therefore,
$$
\lim w_n(t)=w(t).
$$
We also conclude from \eqref{ACJ-2019-11-03-01} that a
finite nontangential limit $w'(t)$ exists and
$$
|w'(t)|\le d.
$$

Conversely, if nontangential limits $w(t)$ and $w'(t)$ exist, $|w(t)|=1$, $w'(t)$ finite, then since
$|w_n(\zeta)|\ge |w(\zeta)|$, we get that $|w_n'(t)|\le |w'(t)|$. Therefore, $|w_n'(t)|$ bounded above and
$$
d\le |w'(t)|.
$$
Thus,
$$
|w'(t)|=\lim |w'_n(t)|.
$$
Since
$$
|w'(t)|=t\frac{w'(t)}{w(t)},
$$
we also get
$$
w'(t)=\lim w'_n(t).
$$
\end{proof}
This theorem is particularly useful when applied to infinite Blaschke products.
\begin{theorem}[Frostman, \cite{Fro}]\label{230526-01}
Let $w(\zeta)$ be an infinite Blaschke product
$$
w(\zeta)=\prod\limits_{k=1}^\infty
\dfrac{\zeta-\zeta_k}{1- \zeta\overline\zeta_k}C_k
,\quad |\zeta|<1.
$$
Let $|t|=1$.
$w(\zeta)$ has a unimodular nontangential boundary value $w(t)$, $|w(t)|=1$, and
$w'(\zeta)$ has a finite nontangential boundary value $w'(t)$
if and only if
$$
\sum\limits_{k=1}^\infty\dfrac{1-|\zeta_k|^2}{|t-\zeta_k|^2}<\infty.
$$
In this case
$$
w(t)=\prod\limits_{k=1}^\infty
\dfrac{t-\zeta_k}{1- t\overline\zeta_k}C_k
$$
and
$$
|w'(t)|=
t\dfrac{w'(t)}{w(t)}=
\sum\limits_{k=1}^\infty\dfrac{1-|\zeta_k|^2}{|t-\zeta_k|^2}.
$$
\end{theorem}
\begin{proof}
This theorem is
Theorem \ref{ACJ-19oct01} applied to the sequence
$$
w_n(\zeta)=\prod\limits_{k=1}^n
\dfrac{\zeta-\zeta_k}{1- \zeta\overline\zeta_k}C_k,
$$
since for $w_n$ we have
$$
w_n(t)=\prod\limits_{k=1}^n
\dfrac{t-\zeta_k}{1- t\overline\zeta_k}C_k
$$
and
$$
|w'_n(t)|=
t\dfrac{w'_n(t)}{w_n(t)}=
\sum\limits_{k=1}^n\dfrac{1-|\zeta_k|^2}{|t-\zeta_k|^2}.
$$
\end{proof}
Next theorem is a version of Theorem \ref{ACJ-03sep01} for functions with positive imaginary part.
\begin{theorem}\label{BK-2020-Jan4-03}
Let function $u$ be analytic in the unit disk with positive imaginary part.
Let $t$ be a point on the unit circle. Assume that $u_0\in\mathbb R$ $(u_0\ne\infty)$. The following are equivalent:
\begin{eqnarray}
&&(1) \quad d_1:={\displaystyle\liminf_{\zeta\to
t}\frac{4\Im u(\zeta)}{1-|\zeta|^2}}<\infty\quad (|\zeta|<1, \zeta {\text\ approaches\ } t
{\text\ in\ an\ arbitrary\ way });\nonumber \\
&&(2) \quad
d_2:={\displaystyle\lim_{\zeta\to t}
\frac{4\Im u(\zeta)}{1-|\zeta|^2}}<\infty
\quad (\zeta {\text\ approaches\ } t
{\text\ nontangentially });\nonumber\\
&&(3;4)\quad \mbox{Nontangential limits exist} \; \;  \nonumber\\
&&\qquad\qquad {\displaystyle\lim_{\zeta\to
t}u(\zeta)=u_0} \; \; \mbox{and} \; \;
u'(t):={\displaystyle\lim_{\zeta\to
t}\frac{u(\zeta)-u_0}{\zeta-t}}\ \mbox{finite}.\nonumber\\
&&(5)\quad \mbox{Nontangential limits exist} \; \; \nonumber\\
&&\qquad\qquad {\displaystyle\lim_{\zeta\to
t}u(\zeta)=u_0} \; \; \mbox{and} \; \; u'_0:={\displaystyle\lim_{\zeta\to
t}u'(\zeta)}\ \mbox{finite}.\nonumber\\
&&(6)\quad \mbox{There exists a constant $d\ge 0$
such that the boundary Schwarz-Pick inequality holds}\nonumber\\
&&\qquad\qquad
\left|\frac{u(\zeta)-u_0}{\zeta-t}\right|^2
\le
d\cdot \frac{\Im u(\zeta)}{1-|\zeta|^2},\quad |\zeta|<1;
\label{BK-2020-Jan4-04}\\
&&\qquad \mbox{\ we denote the smallest constant $d$ that works for \eqref{BK-2020-Jan4-04}
by $d_4$.}\nonumber
\end{eqnarray}
When these conditions hold, we have $u'(t)=u'_0$ and
$$d_1=d_2=d_4=2i t u'_0<\infty.$$
This number is also positive unless $u(\zeta)=u_0$ identically.
\end{theorem}
Next theorem is a version of Theorem \ref{ACJ-19oct01} for functions with
positive imaginary part.
\begin{theorem}\label{ACJ-2019-11-03-02}
Let $u(\z)$, $u_n(\z)$, $|\z|<1$ be analytic functions with positive imaginary part.
Assume that $\Im u_{n+1}(\z)\ge \Im u_n(\z)$ for every $|\z|<1$ and every $n$.
Assume that $u_n(\z)$ converges to $u(\z)$ for every $|\z|<1$.
Let $t$, $|t|=1$. Assume that nontangential boundary values $u_n(t)$,
$u'_n(t)$ exist and that $u_n(t)\in\mathbb R$, $u'_n(t)$
are finite. Under these assumptions nontangential boundary values $u(t)$,
$u'(t)$ exist and  $u(t)\in\mathbb R$, $u'(t)$ is finite
{\bf if and only if}
$u_n'(x_0)$ are bounded above. In this case
$$
u(t)=\lim u_n(t),\quad u'(t)=\lim u_n'(t).
$$
\end{theorem}
\begin{proof}
Proof of this theorem is similar to the proof of Theorem \ref{ACJ-19oct01} and is based on Theorem \ref{BK-2020-Jan4-03}.
Note here that in the present context inequality \eqref{BK-2020-Jan4-04} yields boundedness of $u_n(t)$.
\end{proof}
This theorem is particularly useful when applied to functions with pure point Herglotz measure.
\begin{theorem}\label{230526-02}
Let
$$
u(\z)=i\sum\limits_{k=1}^\infty\dfrac{t_k+\z}{t_k-\z}c_k,\quad |\z|<1,
$$
where $|t_k|=1$, $c_k>0$. Let $|t|=1$, $t\ne t_k$.
Nontangential boundary values $u(t)$,
$u'(t)$ exist and  $u(t)\in\mathbb R$, $u'(t)$ is finite
{\bf if and only if}
\be\label{230526-03}
\sum\limits_{k=1}^\infty\dfrac{c_k}{|t_k-t|^2}<\infty.
\ee
In this case
\be\label{230526-04}
u(t)=i\sum\limits_{k=1}^\infty\dfrac{t_k+t}{t_k-t}c_k
\ee
and
\be\label{230526-05}
u'(t)=2i\sum\limits_{k=1}^\infty\dfrac{t_k}{(t_k-t)^2}c_k
=-2i\overline{t}\sum\limits_{k=1}^\infty\dfrac{c_k}{|t_k-t|^2}.
\ee
\end{theorem}
\begin{proof}
Consider
$$
u_n(\z)=i\sum\limits_{k=1}^n\dfrac{t_k+\z}{t_k-\z}c_k.
$$
Then
$$
u'_n(\z)=i\sum\limits_{k=1}^n\dfrac{2t_k}{(t_k-\z)^2}c_k.
$$
Therefore,
$$
u_n(t)=i\sum\limits_{k=1}^n\dfrac{t_k+t}{t_k-t}c_k
$$
and
$$
u'_n(t)=i\sum\limits_{k=1}^n\dfrac{2t_k}{(t_k-t)^2}c_k
=-2i\overline{t}\sum\limits_{k=1}^n\dfrac{1}{(t_k-t)(\overline{t_k}-\overline{t})}c_k
$$
$$
=-2i\overline{t}\sum\limits_{k=1}^n\dfrac{c_k}{|t_k-t|^2}.
$$
Hence $u'_n(t)$ are bounded if and only if
$$
\sum\limits_{k=1}^\infty\dfrac{c_k}{|t_k-t|^2}<\infty.
$$
Then we apply Theorem \ref{ACJ-2019-11-03-02}.
\end{proof}

%%%\nocite{*}

\bigskip

\bigskip

A. Kheifets, Department of Mathematics and Statistics, University of Massachusetts Lowell, One University Ave.,
Lowell, MA 01854,USA

\emph{E-mail address:} {Alexander\underline{ }Kheifets@uml.edu}


\begin{thebibliography}{99}


%%%\bibitem{BLY-2021}
%%%R. Bessonov, M. Luki\' c, P. Yuditskii, \emph{A comprehensive theory of
%%%reflectionless canonical systems}



%%%%\bibitem{AV}
%%%%D. Alpay, V. Vinnikov, \emph{Indefinite Hardy spaces on finite bordered Riemann surfaces}, J. Funct. Anal. 172 (2000), no. 1, 221--248.
%%%%
%%%%\bibitem{EVY-2019}
%%%%B. Eichinger, T. Vandenboom, P. Yuditskii, \emph{KDV Hierarchy via Abelian Coverings
%%%%and Operator Identities}, Transactions of the American Mathematical Society, Series B
%%%%6 (2019), 1-44.

%%%\bibitem{BV}
%%%J. Ball, V. Vinnikov, \emph{Zero-pole interpolation for matrix meromorphic functions on a compact
%%%Riemann surface and a matrix Fay trisecant identity}, Amer. J. Math. 121 (1999), no. 4, 841--888.

\bibitem{BoKh}
V. Bolotnikov, A. Kheifets,  \emph{A higher order analogue of the Caratheodory-Julia theorem}, Journal of Functional Analysis, 237, no. 1 (2006), 350 - 371

%%%%\bibitem{BS}{A. Borichev, M. Sodin, }\textit{Krein's entire functions and Bernstein approximation problem},
%%%%Illinois Journal of Mathematics,  45, no. 1 (2001), 167--185.

\bibitem{Car}{C. Carath\' eodory, }\textit{Theory of Functions of a Complex Variable},
Engl. Translation, Chelsea Publishing Company, NY, 1960.

%%%\bibitem{CarGam}{L. Carleson, T. Gamelin, }\textit{Complex Dynamics}, Springer-Verlag, 1992.

%%%\bibitem{EYu}{A. Eremenko, P. Yuditskii, }  {\em Comb functions.} Recent advances in orthogonal polynomials,
%%%special functions, and their applications, Contemp. Math.,
%%%{\bf 578} (2012), 99--118.

\bibitem{Fro}
O. Frostman, \emph{Sur les produits de Blaschke}, Kungl. Fysiografiska S\"alskapets I Lund F\"orhandlingar,
12, no.15 (1943), 169 - 182.

%
\bibitem{Hasu} M. Hasumi,
\emph{Hardy Classes on Infintely Connected Riemann Surfaces},
LNM \textbf{1027}, Springer, New York, Berlin, 1983.

%%%\bibitem{Kh}
%%%A. Kheifets, \emph{Abstract interpolation problem and some applications},
%%%Lecture notes given in framework of Holomorphic Spaces semester, fall 1995,
%%%in: Holomorphic Spaces (S. Axler, J. McCarthy, D. Sarason editors), MSRI Publications,
%%%33 (1998) 351-381, Berkeley, California.

%%%%\bibitem{Koo}
%%%%P.~Koosis, \emph{The logarithmic integral. {I}}, Cambridge Studies in Advanced
%%%%  Mathematics, vol.~12, Cambridge University Press, Cambridge, 1998, Corrected
%%%%  reprint of the 1988 original.



%%%\bibitem{Nik}
%%%N.~K. Nikolskii, \emph{Treatise on the shift operator}, A Series of
%%%  Comprehensive Studies in Mathematics, Spriger-Verlag, Berlin Heidelberg New
%%%  York Tokyo, 1986.
  %\MR{0827223}

%%%%%\bibitem{Lev}
%%%%%B. Ya. Levin, \textit{Majorants in classes of subharmonic functions, II, The relation between majorants and conformal mapping,
%%%%%III, The classification of the closed sets on $\bbR$ and the representation of the majorants},
%%%%%Teor. Funktsii Funktsional. Anal. i Prilozhen. 52 (1989), 3--33; English translation: J. Soviet Math. 52 (1990), 3351--3372.

%%%%%\bibitem{LKMV}
%%%%%M. S. Liv\v sic, N. Kravitsky, A. S. Markus, V. Vinnikov,
%%%%%\textit{Theory of Commuting Nonselfadjoint Operators}, Mathematics and Its Applications, 332, Kluwer Academic, Dordrecht (1995)

%%%\bibitem{Mar}
%%%V.~Marchenko, \emph{Sturm-Liouville operators and applications}, Birkh\"auser Verlag, Basel, 1986.
%%%
%%%\bibitem{PF}
%%%B. S. Pavlov, S. I. Fedorov, \textit{Shift group and harmonic analysis on a Riemann surface of genus one},
%%%Algebra i Analiz, 1:2 (1989), 132--168; Leningrad Math. J., 1:2 (1990), 447--490

\bibitem{Kh-Yud-2019} A. Kheifets, P. Yuditskii, \emph{Martin Functions of Fuchsian Groups and Character
Automorphic Subspaces of the Hardy Space in the Upper Half Plane}, Operator Theory: Advances and Applications,
280 (2020), 535 - 581.

\bibitem{Kup-Yud-1997}
S. Kupin, P. Yuditskii, \emph{Analogues of the Nehari and Sarason theorems for character-automorphic
functions and some related questions}, Operator Theory: Advances and Applications 95 (1997), 373-390.

%%%\bibitem{Yud-1997}
%%%P. Yuditskii, \emph{The Character-Automorphic Nehari Problem: Non-Uniqueness Criterion and some
%%%Extremal Solutions}, Journal for Analysis and its Applications
%%%16 (1997) No. 2, 249-261.



\bibitem{Pom} Ch. Pommerenke, \textit{On the Green's function of Fuchsian groups,}
Ann. Acad. Sci. Fenn., 2 (1976), 409-427.

\bibitem{SY}
M. Sodin, P. Yuditskii, \emph{Almost Periodic Jacobi Matrices with Homogeneous Spectrum, Infinite Dimensional Jacobi Inversion, and Hardy Spaces of Character-Automorphic Functions}, Journal of Geometric Analysis
7 (1997) No. 1, 387-435.

%%%\bibitem{VYis}
%%%A. Volberg and P. Yuditskii, \textit{
%%%On the inverse scattering problem for Jacobi matrices with the spectrum on an interval, a finite system of Intervals or a Cantor set of positive length},
%%%Communications in Mathematical Physics, {\bf 226} (2002), 567--605.
%%%
%%%\bibitem{VY}
%%%A. Volberg and P. Yuditskii, \textit{Kotani-Last problem and Hardy spaces on surfaces of Widom type}. Invent. Math., 197 (2014), No. 3, 683-740.
%%%
%%%\bibitem{VYM}
%%%A. Volberg and P. Yuditskii, \textit{Mean type of functions of bounded characteristic and Martin functions in Denjoy domains},
%%%Adv. in Math.,  290 (2016), 860--887.

%
%%%\bibitem{Widom} H.~Widom,
%%%{\it Extremal polynomials associated with a system of curves in the complex plane},
%%%Adv. in Math., {\bf 3} (1969), 127--232.

%
\bibitem{Widom71} H.~Widom,
%%%%{\it $H^p$ sections of vector bundles over Riemann surfaces},
%%%%Ann. of Math., {\bf 94} (1971), 304--324.
{\it The maximum principle for multiple-valued analytic functions},
Acta Math. {\bf 126} (1971), 63--82.


%%%\bibitem{You}
%%%J. You,
%%%extit{Quantitative almost reducibility and its applications}, Proc. Int. Cong. of Math. - 2018, Rio de Janeiro, Vol. 2, 2107--2128.


\end{thebibliography}
 \end{document}